\documentclass[12pt]{article}
\input epsf.tex

\usepackage{amsmath}
\usepackage{amsthm}
\usepackage{amsfonts}
\usepackage{amssymb}
\usepackage{graphicx}
\usepackage{latexsym}

\usepackage{amsmath,amsthm,amsfonts,amssymb,mathrsfs}

\usepackage{epsfig}

\usepackage{amssymb}
\usepackage[all]{xy}
\xyoption{poly}
\usepackage{fancyhdr}
\usepackage{wrapfig}

\theoremstyle{plain}
\newtheorem{thm}{Theorem}[section]

\newtheorem{lem}[thm]{Lemma}
\newtheorem{cor}[thm]{Corollary}

\theoremstyle{definition}
\newtheorem{defn}{Definition}
\theoremstyle{remark}
\newtheorem{remark}{Remark}

\newtheorem{question}{Question}

\newtheorem{notation}{Notation}

\advance \topmargin by -\headheight
\advance \topmargin by -\headsep
\evensidemargin \oddsidemargin
\def\cA{{\cal A}}

\def\Aut{\textrm{Aut}}

\def\cc{{\curvearrowright}}

\def\F{{\mathbb F}}

\def\cJ{{\mathcal{J}}}

\def\MALG{\textrm{MALG}}

\def\cP{{\mathcal P}}

\def\cQ{{\mathcal{Q}}}

\def\R{{\mathbb R}}

\def\chix{{\raise.5ex\hbox{$\chi$}}}

\def\Z{{\mathbb Z}}

\begin{document}
\title{Weak density of orbit equivalence classes of free group actions}
\author{Lewis Bowen\footnote{email:lpbowen@math.utexas.edu} }
\maketitle
\begin{abstract}
It is proven that the orbit-equivalence class of any essentially free probability-measure-preserving action of a free group $G$  is weakly dense in the space of actions of $G$.
 \end{abstract}


\noindent

\section{Introduction}

The results of this paper are motivated by orbit-equivalence theory and weak equivalence of group actions. Let us first recall some terminology before delving into background material.

Let $(X,\mu)$ be a standard non-atomic probability space and $\Aut(X,\mu)$ the group of all measure-preserving automorphisms of $(X,\mu)$ in which we identify two automorphisms if they agree on a conull subset. Let $G$ be a countable group. By an {\em action of $G$} we will mean a homomorphism $a:G \to \Aut(X,\mu)$. In particular, all actions in this paper are probability-measure-preserving. Let  $A(G,X,\mu)$ denote the set of all such actions. It admits a natural topology as follows. First, let us recall that $\Aut(X,\mu)$ is a Polish group with the weak topology (see \S \ref{sec:prelim} for details). We endow the product space $\Aut(X,\mu)^G$ with the product topology and view $A(G,X,\mu)$ as a subspace $A(G,X,\mu) \subset \Aut(X,\mu)^G$ with the induced topology. It is well-known that $A(G,X,\mu)$ is a Polish space \cite{Ke10}.

\subsubsection{Weak containment}

 If $a \in A(G,X,\mu)$ and $T \in \Aut(X,\mu)$, define $a^T \in A(G,X,\mu)$ by $a^T(g)=Ta(g)T^{-1}$. Let $[a]_{MC}=\{a^T:~T \in \Aut(X,\mu)\} \subset A(G,X,\mu)$ be the {\em measure-conjugacy class of $a$}. 
 
  
Given two actions $a,b \in A(G,X,\mu)$ we say {\em $a$ is weakly contained in $b$} (denoted $a\prec b$) if $a$ is contained in the closure of the measure-conjugacy class of $b$ (i.e., $a \in \overline{[b]_{MC}}$). We say {\em $a$ is weakly equivalent to $b$} if $a\prec b$ and $b\prec a$. These notions were introduced by A. Kechris \cite{Ke10} as an analog of weak containment of unitary representations.

 We can also think of weak equivalence as describing the manner in which the Rohlin Lemma fails for non-amenable groups. Recall that the Rohlin Lemma states that any pmp $\Z$-action is approximately periodic. This fundamental fact is critically important in much classical ergodic theory. It has been extended to amenable groups \cite{OW80}. Moreover,  the Rohlin Lemma is essentially equivalent to the statement that if $G$ is infinite and amenable then any essentially free action $a\in A(G,X,\mu)$ weakly contains all actions of $G$ (i.e., $[a]_{MC}$ is dense in $A(G,X,\mu)$) \cite{Ke11}. By contrast, any non-amenable group admits an essentially free strongly ergodic action (e.g., Bernoulli shift actions) \cite{Sc81, KT08}. By definition, the closure of the measure-conjugacy class of a strongly ergodic action cannot contain any non-ergodic action. So each non-amenable group admits at least two non-weakly-equivalent essentially free actions. It is an open problem whether any non-amenable group admits at least two {\em ergodic} non-weakly-equivalent actions. However M. Abert and G. Elek \cite{AE11} made use of profinite actions to show that there is an explicit large family of residually finite groups $G$ that admit an uncountable family of ergodic non-weakly-equivalent actions. This family includes non-abelian free groups.

\subsubsection{Orbit-equivalence}

 We say two actions $a,b \in A(G,X,\mu)$ are {\em orbit-equivalent} if there exists $T\in \Aut(X,\mu)$ which takes orbits to orbits: $T(a(G)x) = b(G)T(x)$ for a.e. $x\in X$. We say that $a\in A(G,X,\mu)$ is {\em essentially free} if for for a.e. $x\in X$ the stabilizer of $x$ in $G$ is trivial: $\{g\in G:~a(g)x=x\} = \{e_G\}$. 
 
If $G$ is amenable then every two essentially free ergodic actions of $G$ are orbit equivalent \cite{OW80}. On the other hand, I. Epstein proved that if $G$ is non-amenable then $G$ admits an uncountable family of essentially free non-orbit-equivalent ergodic pmp actions \cite{Ep09, IKT09}. This followed a series of earlier results that dealt with various important classes of non-amenable groups \cite{GP05, Hj05, Io09, Ki08, MS06, Po06}.  In \cite{IKT09} it shown that essentially free mixing actions of any non-amenable group $G$ cannot be classified by orbit-equivalence up to countable structures.

The main result of this paper shows that, although there are uncountably many essentially free non-orbit-equivalent ergodic pmp actions of any non-abelian free group, the orbit-equivalence class of any such action is dense in the space of all actions.

\subsubsection{Results}
 
 Our main result is:
\begin{thm}\label{thm:main}
Let $G$ be a free group with at most countably many generators. Let $a \in A(G,X,\mu)$ be essentially free and let $[a]_{OE}$ be the set of all actions $b \in A(G,X,\mu)$ which are orbit-equivalent to $a$. Then $[a]_{OE}$ is dense in $A(G,X,\mu)$.
\end{thm}


By contrast we can use rigidity results to show that many groups do not satisfy Theorem \ref{thm:main}. For this purpose, let us recall that if $(K,\kappa)$ is a probability space then any countable group $G$ acts on the product space $(K,\kappa)^G$ by 
$$(gx)(f)=x(g^{-1}f),\quad x\in K^G, g,f \in G.$$
This action is called the {\em Bernoulli shift over $G$} with base space $(K,\kappa)$. 
\begin{thm}\label{thm:counter}
Let $G$ be any countably infinite group satisfying at least one of the following conditions:
\begin{enumerate}
\item $G=G_1\times G_2$ where $G_1,G_2$ are both infinite, $G_1$ is nonamenable and $G$ has no nontrivial finite normal subgroups;
\item $G$ is the mapping class group of the genus $g$ $n$-holed surface for some $(g,n)$ with $3g +n - 4>0$ and $(g,n) \notin \{(1,2),(2,0)\}$;
\item $G$ has property (T) and every nontrivial conjugacy class of $G$ is infinite.
\end{enumerate}
Let $(X,\mu)$ be a standard non-atomic probability space and let $a \in A(\Gamma,X,\mu)$ be isomorphic to the Bernoulli action $G \cc ([0,1],\lambda)^G$ where $\lambda$ is the Lebesgue measure on the unit interval $[0,1]$. Then $[a]_{OE}$ is not dense in $A(G,X,\mu)$.
\end{thm}

Before proving this, we need a lemma.
\begin{defn}
Let $a\in A(G,X,\mu)$ and $\alpha \in \Aut(G)$. Observe that the composition $a \circ \alpha \in A(G,X,\mu)$. We say that two actions  $a,b \in A(G,X,\mu)$ are {\em conjugate up to automorphisms} if there exists $\alpha \in \Aut(G)$ and $T \in \Aut(X,\mu)$ such that $b = (a \circ \alpha)^T$. 
\end{defn}

\begin{lem}
Let $G$ be any countable group, $(K,\kappa)$ a standard probability space and $G \cc^a (K,\kappa)^G$ the Bernoulli shift action. (So $(gx)(f)=x(g^{-1}f)$ for $x\in K^G$ and $g,f \in G$). Then any action of $G$ which is conjugate to $a$ up to automorphisms is measurably conjugate to $a$.
\end{lem}

\begin{proof}
Suppose $\alpha \in \Aut(G)$. It suffices to show  that $a$ is measurable conjugate to $a \circ \alpha$. For this purpose, define $T:K^G \to K^G$ by $T(x)(g)=x(\alpha^{-1}(g))$. Then for any $g,f \in G$ and $x\in K^G$,
$$ T(f x)(g) = (fx)(\alpha^{-1}(g)) = x(f^{-1}\alpha^{-1}(g)) = x(\alpha^{-1}( \alpha(f^{-1})g)) = (Tx)(\alpha(f^{-1})g) = \alpha(f)(Tx)(g).$$
This shows that $T$ intertwines $a$ and $a\circ \alpha$ as required.
\end{proof}

\begin{proof}[Proof of Theorem \ref{thm:counter}]
Using the previous lemma and \cite[Corollary 1.3]{Po08}, \cite[Theorem 1.4]{Ki08} and \cite[Corollary 0.2]{Po06} we obtain that 
if $G$ and $a$ are as above then $[a]_{OE}=[a]_{MC}$. Moreover $a$ is strongly ergodic \cite{KT08}. So there does not exist any non-ergodic actions in the closure of its measure-conjugacy class. In particular $[a]_{MC}$ is not dense in $A(G,X,\mu)$.
\end{proof}

\begin{remark}
Theorem \ref{thm:main} and the upper semi-continuity of cost \cite[Theorem 10.12]{Ke10} imply that finitely generated free groups have fixed price, a fact originally obtained by Gaboriau \cite{Ga00}.
\end{remark}

\begin{remark}
 Because free groups are residually finite and therefore sofic, Theorem \ref{thm:main} implies that the orbit-equivalence relation of every essentially free $a\in A(\Gamma,X,\mu)$ is sofic. This fact was first obtained in \cite{EL10} (it can also be obtained as a consequence of property MD for free groups \cite{Ke11} which was discovered earlier in a different context in \cite{Bo03}). A new proof of this appears in \cite{BLS13}.
\end{remark}

\begin{question}
Which groups $\Gamma$ satisfy the conclusion of Theorem \ref{thm:main}? For example, do all strongly treeable groups satisfy this conclusion? Does $PSL(2,\R)$ satisfy the conclusion? 
\end{question}


\begin{question}
Are orbit-equivalence classes meager? That is, is the set $[a]_{OE}$ from Theorem \ref{thm:main}  meager in $A(G,X,\mu)$? If so, then combined with ideas and results of \cite{IKT09} it should be possible to prove that if $G$ is a nonabelian free group then for any comeager subset $Y \subset A(G,X,\mu)$ it is not possible to classify actions in $Y$ by orbit-equivalence up to countable structures.
\end{question}


\subsection{A special case}

To give the reader a feeling for the proof of Theorem \ref{thm:main}, we show how to quickly prove a special case.

\begin{thm}\label{thm:easy}
Let $G$ be a non-abelian finitely generated free group and $a \in A(G,X,\mu)$ be essentially free. Let $S \subset G$ be a free generating set. Suppose the for every $s\in S$, the automorphism $a(s) \in \Aut(X,\mu)$ is ergodic. Then $[a]_{OE}$ is dense in $A(G,X,\mu)$.
\end{thm}

\begin{lem}\label{lem:FW}
Suppose that $T \in \Aut(X,\mu)$ is ergodic, $\epsilon>0$ and $\{C_i\}_{i<k}, \{D_i\}_{i<k}$ two measurable partitions of $X$ such that for each $i < k$, $C_i$ and $D_i$ have the same measure. Then there is a $T' \in \Aut(X,\mu)$ with the same orbits as $T$ such that for all $i$ the measure of $T'(C_i)\vartriangle D_i$ is less than $\epsilon$.
\end{lem}
\begin{proof}
This is \cite[Lemma 7]{FW04}.
\end{proof}

\begin{proof}[Proof of Theorem \ref{thm:easy}]
Let $b \in A(G,X,\mu)$. We will show that $b \in \overline{[a]_{OE}}$. By Lemma \ref{lem:generator} below, it suffices to show that if $\{C_i\}_{i<k}$ is a measurable partition of $X$ and $\epsilon>0$ then there exists a measurable partition $\{D_i\}_{i<k}$ of $X$ and an action $a' \in [a]_{OE}$ such that
\begin{eqnarray}\label{eqn:suff}
| \mu(C_i \cap b_s C_j) - \mu(D_i \cap a'_s D_j)| < \epsilon
\end{eqnarray}
for every $s\in S$ and $1\le i,j < k$ (where for example $b_s=b(s)$).

By Lemma \ref{lem:FW} for every $s\in S$ there is an automorphism $a'_s \in \Aut(X,\mu)$ with the same orbits as $a_s$ such that 
$$\mu( a'_s(C_i) \vartriangle b_s(C_i)) < \epsilon.$$
Therefore, equation (\ref{eqn:suff}) holds with $D_i=C_i$ for all $i$. It is easy to verify that $a'$ is orbit-equivalent to $a$ (indeed $a'$ has the same orbits as $a$).
\end{proof}

The conclusion of Lemma \ref{lem:FW} does not hold in general if $T$ is non-ergodic. In order to prove Theorem \ref{thm:main} we will show instead that  if the sets $\{C_i\}_{i<k}$ are sufficiently equidistributed with respect to the ergodic decomposition of $T$ then we can find an automorphism $T'$ with the same orbits as $T$ such that the numbers $\mu(C_i \cap T' C_j)$ are close to any pre-specified set of numbers satisfying the obvious restrictions.

{\bf Acknowledgements}. Thanks to Robin Tucker-Drob for pointing me to \cite[Lemma 7]{FW04}. I am partially supported by NSF grant DMS-0968762 and NSF CAREER Award DMS-0954606.


\section{The weak topology}\label{sec:prelim}

Here we review the weak topology and obtain some general results regarding weak containment. So let $(X,\mu)$ be a standard non-atomic probability space. The {\em measure algebra} of $\mu$, denoted $\MALG(\mu)$, is the collection of all measurable subsets of $X$ modulo sets of measure zero. There is a natural distance function on the measure-algebra defined by
$$d( A,B) = \mu(A \vartriangle B)$$
for any $A,B \in \MALG(\mu)$. Because $\mu$ is standard, there exists a dense sequence $\{A_i\}_{i=1}^\infty \subset \MALG(\mu)$. Using this sequence we define the weak-distance between elements $T,S \in \Aut(X,\mu)$ by:
$$d_w(T,S) = \sum_{i=1}^\infty 2^{-i} \mu(TA_i \vartriangle SA_i).$$
The topology induced by this distance is called the {\em weak topology}. While $d_w$ depends on the choice of $\{A_i\}_{i=1}^\infty$, the topology on $\Aut(X,\mu)$ does not depend on this choice. 

Let $G$ be a countable group. Recall that $A(G,X,\mu)$ denotes the set of all homomorphisms $a:G \to \Aut(X,\mu)$. We may view $A(G,X,\mu)$ as a subset of the product space $\Aut(X,\mu)^G$ from which it inherits a Polish topology \cite{Ke10}.

\begin{notation}
If $v \in A(G,X,\mu)$ and $g \in G$ then we write $v_g=v(g)$.
\end{notation}

\begin{lem}\label{lem:Kechris}
Let $G$ be a countable group. Let $v \in A(G,X,\mu)$ and $W \subset A(G,X,\mu)$. Then $v$ is in the closure $\overline{W}$ if and only if: for every $\epsilon>0$, for every finite Borel partition $\cP=\{P_1,\ldots, P_n\}$ of $X$ and every finite set $F \subset G$ there exists $w\in W$ and a finite Borel partition $\cQ=\{Q_1,\ldots, Q_n\}$ of $X$ such that
$$ | \mu(P_i \cap v_g P_j) - \mu(Q_i \cap w_g Q_j)| < \epsilon$$
for every $g \in F$ and $1\le i,j\le n$.
\end{lem}

\begin{proof}
This is essentially the same as \cite[Proposition 10.1]{Ke10}. It also follows from \cite[Theorem 1]{CKT12}. 
\end{proof}

\begin{cor}\label{cor:fin-gen}
In order to prove Theorem \ref{thm:main}, it suffices to prove the special case in which $G$ is finitely generated.
\end{cor}

\begin{proof}
 Let $G$ be a countably generated free group with free generating set $S=\{s_1,s_2,\ldots \} \subset G$. Let $a,b \in A(G,X,\mu)$ and suppose $a$ is essentially free. Let $\epsilon>0$, $\cP=\{P_1,\ldots, P_n\}$ be a Borel partition of $X$ and $F \subset G$ be finite. By Lemma \ref{lem:Kechris} it suffices to show there exists $a'\in [a]_{OE}$ and a finite Borel partition $\cQ=\{Q_1,\ldots, Q_n\}$ of $X$ such that
 \begin{eqnarray}\label{eqn:a'}
  | \mu(P_i \cap b_g P_j) - \mu(Q_i \cap a'_g Q_j)| < \epsilon
  \end{eqnarray}
for every $g \in F$ and $1\le i,j\le n$.
Let $G_n<G$ be the subgroup generated by $\{s_1,\ldots, s_n\}$. Choose $n$ large enough so that $F \subset G_n$. Because we are assuming Theorem \ref{thm:main} is true for finitely generated free groups, there exists an action $a'' \in A(G_n,X,\mu)$ orbit-equivalent to $a|_{G_n}$ such that 
\begin{eqnarray}\label{eqn:a''}
 | \mu(P_i \cap b_g P_j) - \mu(Q_i \cap a''_g Q_j)| < \epsilon
 \end{eqnarray}
for every $g \in F$ and $1\le i,j\le n$. By definition of orbit-equivalence, there exists an automorphism $T \in \Aut(X,\mu)$ such that $a''$ and $(a|_{G_n})^T$ have the same orbits. 

Define $a' \in A(G,X,\mu)$ by $a'(s_i) = a''(s_i)$ if $1\le i \le n$ and $a'(s_i) = Ta(s_i)T^{-1}$ for $i>n$. Then clearly $a'$ is orbit-equivalent to $a$ and $a'$ satisfies (\ref{eqn:a'}) because of (\ref{eqn:a''}).

\end{proof}

The next result implies that we can replace the finite set $F\subset G$ appearing in the lemma above with a fixed generating set $S\subset G$. This is crucial to the whole approach because it allows us to reduce Theorem \ref{thm:main} from a problem about actions of the free group to a problem about actions of the integers.

\begin{lem}\label{lem:generator}
Let $G$ be a group with a finite symmetric generating set $S$. Let $v \in A(G,X,\mu)$ and $W \subset A(G,X,\mu)$. Suppose that for every $\epsilon>0$ for every finite Borel partition $\cP=\{P_1,\ldots, P_n\}$ of $X$ there exists $w\in W$ and a finite Borel partition $\cQ=\{Q_1,\ldots, Q_n\}$ of $X$ such that
$$ | \mu(P_i \cap v_s P_j) - \mu(Q_i \cap w_s Q_j)| < \epsilon$$
for every $s\in S$ and $1\le i,j\le n$. Then $v \in \overline{W}$.
\end{lem}

\begin{proof}
Let $\epsilon>0$, $\cP=\{P_1,\ldots,P_n\}$ be a Borel partition of $X$ and $F \subset G$ be a finite set. By Lemma \ref{lem:Kechris} it suffices to show that  there exists $w\in W$ and a finite Borel partition $\cQ=\{Q_1,\ldots, Q_n\}$ of $X$ such that
\begin{eqnarray}\label{eqn:Q}
 | \mu(P_i \cap v_g P_j) - \mu(Q_i \cap w_g Q_j)| < \epsilon
 \end{eqnarray}
for every $g \in F$ and $1\le i,j\le n$.

In order to do this, we may assume that for some integer $r\ge 0$, $F=B(e,r)$ is the ball of radius $r$ centered at the identity in $G$ with respect to the word metric $d_S(\cdot,\cdot)$ induced by the generating set $S$. 

Let $\cP' = \bigvee_{g\in F} g\cP$ be the common refinement of the partitions $\{g\cP:~g \in F\}$. By hypothesis, there exists a partition $\cQ'$ of $X$, a bijection $\beta: \cP' \to \cQ'$ and an action $w\in W$ such that
\begin{eqnarray}\label{eqn:weak}
 | \mu(P' \cap v_s P'') - \mu( \beta(P') \cap w_s \beta(P''))| < \epsilon |\cP'|^{-2} |F|^{-1}/4
 \end{eqnarray}
for every $P', P'' \in \cP'$ and $s\in S$.

Let $\Sigma(\cP')$ denote the sigma-algebra generated by $\cP'$ (and define $\Sigma(\cQ')$  similarly). There is a unique boolean-algebra homomorphism from $\Sigma(\cP')$ to $\Sigma(\cQ')$ extending $\beta$. We also let $\beta$ denote this homomorphism.

Let $\cQ=\{\beta(P):~P \in \cP\}$. It is immediate that $\cQ$ is a finite Borel partition of $X$. We will show that it satisfies (\ref{eqn:Q}). 

\noindent {\bf Claim 1}. $|\mu(\beta(P)) - \mu(P)| < \epsilon/2$ for every $P \in \Sigma(\cP')$.

\begin{proof}[Proof of Claim 1]
Let $s\in S$. By (\ref{eqn:weak})
\begin{eqnarray*}
|\mu(\beta(P)) - \mu(P)| &\le & \sum_{P', P''}  | \mu(P' \cap v_s P'') - \mu( \beta(P') \cap w_s \beta(P''))| < \epsilon/2
\end{eqnarray*} 
where the sum is over all $P', P'' \in \cP'$ with $P' \subset P$.
\end{proof}

\noindent {\bf Claim 2}. $\mu( \beta(v_gP) \vartriangle w_g\beta(P)) \le \epsilon |g|/(2|F|)$ for all $P \in \cP$ and $g\in F$ where $|g|$ denotes the word length of $g$. Moreover equality holds only in the case $|g|=0$.

\begin{proof}[Proof of Claim 2]
We prove this by induction on the word length $|g|$. It is obviously true when $|g|=0$. So we assume there is an integer $m\ge 0$ such that the statement is true for all $g$ with $|g|\le m$. Now suppose that $|g|=m+1$ and $g\in F$. Then $g=sh$ for some $h\in F$ and $s\in S$ such that $|h|=m$. By induction,
\begin{eqnarray*}
\mu( \beta(v_gP) \vartriangle w_g\beta(P)) &=& \mu( \beta(v_{sh}P) \vartriangle w_{sh}\beta(P)) \\
&\le&\mu( \beta(v_{sh}P) \vartriangle w_s\beta(v_hP))  +  \mu( w_s\beta(v_hP) \vartriangle w_{sh}\beta(P)) \\
&=& \mu( \beta(v_{sh}P) \vartriangle w_s\beta(v_hP))  +  \mu( \beta(v_hP) \vartriangle w_h\beta(P)) \\
&\le& \mu( \beta(v_{sh}P) \vartriangle w_s\beta(v_hP))  +  \epsilon |h|/(2|F|). 
\end{eqnarray*}
Next we observe that
\begin{eqnarray*}
\mu(\beta(v_{sh}P) \vartriangle w_s\beta(v_hP)) = \sum_{P_1,P_2} \mu(\beta(P_1) \cap w_s \beta(P_2)) + \sum_{P_3,P_4} \beta(P_3 \cap w_s\beta(P_4))
\end{eqnarray*}
where the first sum is over all $P_1,P_2 \in \cP'$ such that $P_1 \subset v_{sh}P$ and $P_2 \cap v_hP = \emptyset$ while the second sum is over all $P_3,P_4 \in \cP'$ such that $P_3 \cap v_{sh}P = \emptyset$ and $P_4 \subset v_hP$. 

By (\ref{eqn:weak}) if $(i,j)=(1,2)$ or $(i,j)=(3,4)$ as above then
\begin{eqnarray*}
\mu(\beta(P_i) \cap w_s \beta(P_j)) < \mu(P_i \cap v_s P_j) +  \epsilon |\cP'|^{-2}/(4|F|) = \epsilon |\cP'|^{-2}/(4|F|).
\end{eqnarray*}
Therefore,
\begin{eqnarray*}
\mu(\beta(v_{sh}P) \vartriangle w_s\beta(v_hP)) < \epsilon/(2|F|). 
\end{eqnarray*}
This implies the claim.
\end{proof}

Next we verify (\ref{eqn:Q}) with $Q_i = \beta(P_i)$: 
\begin{eqnarray*}
 | \mu(P_i \cap v_g P_j) - \mu(Q_i \cap w_g Q_j)| &=&  | \mu(P_i \cap v_g P_j) - \mu(\beta(P_i) \cap w_g \beta(P_j))| \\
 &<&  | \mu(P_i \cap v_g P_j) - \mu(\beta(P_i) \cap  \beta(v_g P_j))| + \epsilon |g|/(2|F|)\\
  &=& | \mu(P_i \cap v_g P_j) - \mu(\beta(P_i \cap  v_g P_j))| + \epsilon |g|/(2|F|) \\
  &<& \epsilon/2 + \epsilon |g|/(2|F|)  \le \epsilon.
  \end{eqnarray*}  
The first inequality follows from Claim 2 and the second inequality from Claim 1.


\end{proof}

\section{A combinatorial lemma}

The next lemma will be used to rearrange partial orbits of a single transformation. Roughly speaking it states that any partial orbit which is roughly equidistributed with respect to some partition can be rearranged so as to approximate the local statistics of any given Markov chain on the partition.

\begin{lem}\label{lem:combinatorial}
Let $\cA$ be a finite set, $\pi$ be a probability distribution on $\cA$ and $\cJ$ be a self-coupling of $\pi$ (so $\cJ$ is a probability distribution on $\cA\times \cA$ such that the projection of $\cJ$ to either factor is $\pi$). Let $0<\epsilon<1$ and $N>0$ be an integer and $\phi:\{1,\ldots, N\} \to \cA$ a map such that if $\pi'$ is the empirical distribution of $\phi$ then $\| \pi' - \pi\|_\infty < \epsilon$. (By definition, $\pi' = \phi_* u_N $ where $u_N$ is the uniform probability distribution on $\{1,\ldots, N\}$). We assume $\min_{a,b \in \cA} \cJ(a,b) >2 |\cA|\epsilon + |\cA|^2/N$.

Then there exists a bijection $\sigma=\sigma_\phi:\{1,\ldots, N-1\} \to \{2,\ldots, N\}$ such that if $\Gamma(\sigma)$ is the graph with vertices $\{1,\ldots,N\}$ and edges $\{ (i,\sigma(i)):~ 1\le i \le N-1\}$ then
\begin{itemize}
\item $\Gamma(\sigma)$ is connected (so it is isomorphic to a line graph)
\item if $\Phi_\sigma:\{1,\ldots, N-1\} \to \cA \times \cA$ is the map $\Phi_\sigma(i) = ( \phi(i), \phi(\sigma(i)))$ and $\cJ_\sigma=(\Phi_\sigma)_*u_{N-1}$ is the empirical distribution of $\Phi_\sigma$ then
$$\| \cJ_\sigma - \cJ\|_\infty  < 2|\cA|\epsilon + 3|\cA|^2/N.$$
\end{itemize}
\end{lem}

\begin{proof}

\noindent {\bf Claim 1}. There exists a self-coupling $\cJ'$ of $\pi'$ such that 
$$\| \cJ' - \cJ \|_\infty < 2 |\cA|\epsilon + |\cA|^2/N$$
and $\cJ'$ takes values only in $\Z[1/N]$. 
\begin{proof}[Proof of Claim 1]
Let $a\in \cA$. For $b,c\in \cA \setminus \{a\}$, let $\cJ'(b,c)$ be the closest number in $\Z[1/N]$ to $\cJ(b,c)$. Define 
\begin{eqnarray*}
\cJ'(a,c) &=& \pi'(c) - \sum_{t \in \cA\setminus \{a\}} \cJ'(t,c) \\
\cJ'(b,a) &=& \pi'(b) - \sum_{t \in \cA\setminus \{a\}} \cJ'(b,t) \\
\cJ'(a,a) &=& \pi'(a) - \sum_{t \in \cA \setminus \{a\}} \cJ'(a,t) = \pi'(a) - \sum_{t \in \cA \setminus \{a\}} \cJ'(t,a).
\end{eqnarray*}
It is straightforward to check that 
$$\| \cJ' - \cJ \|_\infty < \epsilon +|\cA|( \epsilon + |\cA|/N) \le 2 |\cA|\epsilon + |\cA|^2/N.$$
Because $\min_{a,b \in \cA} \cJ(a,b) >2 |\cA|\epsilon + |\cA|^2\epsilon/N$, this implies $\cJ'$ is positive everywhere. So it is a self-coupling of $\pi'$.
 \end{proof}

\noindent {\bf Claim 2}. There exists a bijection $\tau:\{1,\ldots, N-1\} \to \{2,\ldots, N\}$ such that if $\Phi_\tau:\{1,\ldots, N\} \to \cA \times \cA$ is the map $\Phi_\tau(i) = ( \phi(i), \phi(\tau(i)))$ and $\cJ_\tau= (\Phi_\tau)_*u_N$ is the empirical distribution of $\Phi_\tau$ then $\|\cJ_\tau-\cJ'\|_\infty \le 1/N$.
\begin{proof}[Proof of Claim 2]
Because $\cJ'$ is a self-coupling of $\pi'$ taking values in $\Z[1/N]$ there exist partitions $\cP=\{P_{a,b}\}_{a,b\in \cA},\cQ=\{Q_{a,b}\}_{a,b\in \cA}$ of $\{1,\ldots, N\}$ such that:
\begin{itemize}
\item $|P_{a,b}| = |Q_{a,b}|  = N \cJ'(a,b)$ for every $a,b \in \cA$;
\item $P_{a,b} \subset \phi^{-1}(a)$ and $Q_{a,b} \subset \phi^{-1}(b)$ for every $a,b \in \cA$.
\end{itemize}
Next we choose bijections $\beta_{a,b}:P_{a,b} \to Q_{a,b}$ for all $a,b \in \cA$. Define $\tau:\{1,\ldots, N-1\} \to \{2,\ldots, N\}$ by $\tau(i) = \beta_{a,b}(i)$ if $i\in P_{a,b}$ and $\beta_{a,b}(i) \ne 1$. If $i \in P_{a,b}$ (for some $a,b$) and $\beta_{a,b}(i)=1$ then we define $\tau(i)=N$. This satisfies the claim.
\end{proof}
Let $\tau:\{1,\ldots, N-1\} \to \{2,\ldots, N\}$ be a bijection satisfying the conclusion of Claim 2 with the property that the number of connected components of the graph $\Gamma(\tau)$ is as small as possible given that $\tau$ satisfies Claim 2. 

\noindent {\bf Claim 3}. $\Gamma(\tau)$ has at most $|\cA|^2$ connected components.

\begin{proof}[Proof of Claim 3]
To obtain a contradiction, suppose $\Gamma(\tau)$ has more than $|\cA|^2$ connected components. Then there exists $1\le i <j \le N-1$ such that $i$ and $j$ are in different connected components of $\Gamma(\tau)$, $\phi(i)=\phi(j)$ and $\phi(\tau(i))=\phi(\tau(j))$.  Let us define $\tau':\{1,\ldots, N-1\} \to \{2,\ldots, N\}$ by 
\begin{displaymath}
\tau'(k) = \left\{ \begin{array}{ll}
\tau(k) & k \notin \{i,j\} \\
\tau(j) & k=i \\
\tau(i) & k=j 
\end{array}\right.
\end{displaymath}
Observe that $\tau'$ also satisfies Claim 2 and $\Gamma(\tau')$ has one fewer connected component than $\Gamma(\tau)$ contradicting the choice of $\Gamma(\tau)$.
\end{proof}

Let $1\le i_1 < i_2 < \ldots < i_k \le N$ be a maximal set of indices such that for $t \ne s$, $i_t$ and $i_s$ are in different connected components of $\Gamma(\tau)$. Define the bijection $\sigma:\{1,\ldots,N-1\} \to \{2,\ldots, N\}$ by
\begin{displaymath}
\sigma(t) = \left\{ \begin{array}{ll}
\tau(t) & t \notin \{i_1,\ldots, i_k\} \\
\tau(i_{s+1}) & t=i_s~\textrm{ (indices mod $k$)}
\end{array}\right.
\end{displaymath}
Observe that $\Gamma(\sigma)$ is connected and, since $k\le |\cA|^2$,
$$    \left|\cJ_\sigma(a,b) - \cJ_\tau(a,b) \right|  = \left|\frac{ \#\{ 1\le i \le N-1:~ \phi(i) =  a, \phi(\sigma(i)) = b\} }{N} - \cJ_\tau(a,b) \right| \le |\cA|^2/N.$$
This implies the lemma.
\end{proof}

\section{Proof of Theorem \ref{thm:main}}

From here on, it will be convenient to work with observables instead of partitions. So instead of considering partitions $\cP$ of a space $X$ we consider measurable maps $\phi:X \to \cA$ where $\cA$ is a finite set. Of course, $\{\phi^{-1}(\{a\}):~a\in \cA\}$ is the partition of $X$ represented by $\phi$.

\begin{lem}\label{lem:hard}
Let $(X,\mu)$ be a standard probability space and $T \in \Aut(X,\mu)$ be aperiodic. Let $\psi: X \to \cA$ be a measurable map into a finite set. Let $\mu = \int \nu~d\omega(\nu)$ be the ergodic decomposition of $\mu$ with respect to $T$ (so $\omega$ is a probability measure on the space of   $T$-invariant ergodic Borel probability measures on $X$). Suppose that for some $1/6>\epsilon>0$,
$$\omega\Big(\big\{ \nu:~  \|\psi_*\nu - \psi_*\mu \|_\infty > \epsilon \big\}\Big) < \epsilon.$$
Suppose also that $\cJ$ is a self-coupling of $\psi_*\mu$ (i.e. $\cJ$ is a probability measure on $\cA \times \cA$ whose projections are both equal to $\psi_*\mu$) and
$$\min_{a,b \in \cA} \cJ(a,b) >2|\cA|\epsilon.$$
Then there exists $T' \in \Aut(X,\mu)$ such that $T$ and $T'$ have the same orbits (a.e.) and if $\Phi: X \to \cA \times \cA$ is the map $\Phi(x)=(\psi(x),\psi(T'x))$ then 
$$ \| \Phi_* \mu  - \cJ\|_\infty \le 9|\cA|\epsilon.$$
\end{lem}

\begin{proof}


By the pointwise ergodic theorem, there exists a Borel set $X' \subset X$ and an integer $M>0$ such that
\begin{itemize}
\item $\mu(X')>1-\epsilon$
\item for every $x\in X'$, every $a\in \cA$ and every $K_1,K_2 \ge M$,
$$\left| \frac{ \#\{ -K_1 \le j \le K_2:~ \psi(T^j x) = a \} }{K_1+K_2+1} - \psi_*\mu(a) \right| < \epsilon.$$
\end{itemize}
Without loss of generality, we may assume $M$ is large enough so that
$$\min_{a,b \in \cA} \cJ(a,b) >2 |\cA|\epsilon + |\cA|^2/(2M+1)$$
and $3|\cA|^2/(2M+1) < \epsilon$. 



Let $Y \subset X$ be a complete section with $\mu(Y)\le \epsilon/(2M+1)$. By a complete section we mean that for $\mu$-a.e. $x\in X$ the orbit of $x$ intersects $Y$ nontrivially. The existence of such a complete section is proved in \cite[Chapter II, Lemma 6.7]{KM04}. Without loss of generality, we will assume that if $y\in Y$ then $Ty \notin Y$.

For any integer $N\ge 1$ let us say that a map $\eta:\{1,\ldots, N\} \to \cA$ is {\em $\epsilon$-good} if 
$$\| \eta_*u_N - \psi_*\mu \|_\infty < \epsilon$$
and 
$$\min_{a,b \in \cA} \cJ(a,b) >2 |\cA|\epsilon + |\cA|^2/N$$
where $u_N$ is the uniform probability measure on $\{1,\ldots, N\}$. For each  $\epsilon$-good map $\eta:\{1,\ldots,N\} \to \cA$ choose a map $\sigma_\eta:\{1,\ldots, N-1\} \to \{2,\ldots, N\}$ as in Lemma \ref{lem:combinatorial}.

For $x\in X$, let $\alpha(x)$ be the smallest nonnegative integer such that $T^{-\alpha(x)}x \in Y$, let $\beta(x)$ be the smallest nonnegative integer such that $T^{\beta(x)}x \in Y$ and let $\psi_x:\{1,\ldots, \alpha(x) + \beta(x)+1\} \to \cA$ be the map $\psi_x(j) = \psi(T^{j-\alpha(x)-1} x)$. So
$$\psi_x = \big(\psi(T^{-\alpha(x)}x),  \psi(T^{-\alpha(x)+1}x), \ldots, \psi(T^{\beta(x)}x) \big).$$
Note that $\psi_x$ is $\epsilon$-good if $x \in X'$, $\alpha(x) \ge M$ and $\beta(x) \ge M$. In this case, let $\sigma_x=\sigma_{\psi_x}$ and $\cJ_x=\cJ_{\sigma_x}$ (with notation as in Lemma \ref{lem:combinatorial}). 

Now we can define $T'$ as follows. If either $x \in Y$ or $\psi_x$ is not $\epsilon$-good then we define $T'x=Tx$. Otherwise we set $T'x = T^{\sigma_x(\alpha(x)+1)-\alpha(x)-1}(x)$.  Because each $\sigma_x$ is a bijection and the graph $\Gamma(\sigma_x)$ is connected it immediately follows that $T'$ and $T$ have the same orbits. 

By Kac's Theorem (see for example \cite[Theorem 4.3.4]{Do11}),
$$\int_Y \beta(Tx)+1~d\mu(x) = 1.$$
Therefore the set $Y'$ of all $x \in Y$ such that $\beta(Tx)+1 \ge 2M+1 $ satisfies
$$\int_{Y'} \beta(Tx)+1~d\mu(x) = 1 - \int_{Y\setminus Y'} \beta(Tx)+1~d\mu(x) \ge 1 - \mu(Y)(2M+1) \ge 1-\epsilon.$$
Let $X''$ be the set of all $T^jx$ for $x \in Y'$ and $M \le j \le \beta(Tx)-M$. Then 
$$\mu(X'') = \int_{Y'}( \beta(Tx) -2M +1) ~d\mu(x) \ge 1-\epsilon - \mu(Y')(2M) \ge 1-2\epsilon.$$
Let $X''' = X' \cap X''$. So $\mu(X''') \ge 1-3\epsilon$. Observe that if $x\in X'''$ then $\alpha(x)\ge M$ and $\beta(x)\ge M$ so $\psi_x$ is $\epsilon$-good. Finally, let $Y''$ be the set of all $y\in Y'$ such that $T^jy \in X'''$ for some $1\le j \le \beta(Tx)+1$. Then
$$\int_{Y''} \beta(Tx)+1~d\mu(x) \ge \mu(X''') \ge 1-3\epsilon.$$
Moreover, if $y \in Y''$ then $\psi_{Ty}$ is $\epsilon$-good (this uses our hypothesis that if $y\in Y'' \subset Y$ then $Ty \notin Y$) . Let $Z$ be the set of all $T^jy$ for $y \in Y''$ and $0\le j \le \beta(Ty)$. So
$$\mu(Z) = \int_{Y''} \beta(Tx)+1 ~d\mu(x) \ge 1-3\epsilon.$$

Recall that $\Phi:X \to \cA \times \cA$ is defined by $\Phi(x) = (\psi(x), \psi(T'x))$. Let $\mu_{Z}$ denote the unnormalized restriction of $\mu$ to $Z$. Then
\begin{eqnarray*}
\| \Phi_*\mu  - \cJ \|_\infty &\le& \| \Phi_*\mu - \Phi_* \mu_{Z} \|_\infty + \|\Phi_*\mu_{Z} - \cJ \|_\infty\\
&\le& 3\epsilon + \|\Phi_*\mu_{Z} - \cJ \|_\infty.
\end{eqnarray*}
Next observe that 
$$\Phi_*\mu_Z = \int_{Y''} [\beta(Tx)+1] \cJ_{Tx}~d\mu(x).$$
By Lemma \ref{lem:combinatorial}, for $x\in Y''$
$$\|\cJ_{Tx} - \cJ\|_\infty \le 2|\cA|\epsilon + 3|\cA|^2/(2M+1) \le 3|\cA|\epsilon.$$
So
$$ \|\Phi_*\mu_{Z} - \cJ \|_\infty \le (3|\cA|\epsilon)/(1-3\epsilon) \le 6|\cA|\epsilon$$
(because $\epsilon<1/6$) and therefore,
$$\| \Phi_*\mu  - \cJ \|_\infty  \le 3\epsilon + 6|\cA|\epsilon \le 9|\cA|\epsilon.$$

\end{proof}

In the next lemma, we prove the existence of a ``good observable'' $\psi$.



\begin{lem}\label{lem:good-partition}
Let $G$ be a countable group and $S \subset G$ a finite set of elements of infinite order. Let $T\in A(G,X,\mu)$ be an essentially free action of $G$. For each $s \in S$ let $\mu = \int \nu ~d\omega_s(\nu)$ be the ergodic decomposition of $\mu$ with respect to $T_s$ (so $\omega_s$ is a Borel probability measure on the space of all $T_s$-invariant ergodic Borel probability measures). 

Let $\pi$ be a probability measure on a finite set $\cA$. Also let $0<\epsilon<1/2$. Then there exists a measurable map $\psi:X \to \cA$ such that for every $s \in S$,
$$\omega_s(\{ \nu:~ \|\psi_*\nu - \pi \|_\infty > 3\epsilon \}) < \epsilon.$$
\end{lem}

\begin{proof}
Given a measurable map $\psi:X \to \cA$, an element $s\in S$ and an integer $N\ge 1$ let $\psi_{s,N}:X \to \cA^N$ be the map
$$\psi_{s,N}(x) = (\psi(T_s x), \psi(T^2_s x), \ldots, \psi(T^N_s x) ).$$
 According to Abert-Weiss \cite{AW11}, the action $T$ weakly contains the Bernoulli shift action $G \cc (\cA,\pi)^G$. This immediately implies that for any integer $N\ge 1$ there exists a measurable map $\psi:X \to \cA$ such that for every $s \in S$
 $$\| (\psi_{s,N})_*\mu - \pi^N \|_\infty < \frac{\epsilon^2}{2}.$$
 Given a sequence $y \in \cA^N$, let $E[y]$ denote its empirical distribution. To be precise, $E[y]$ is the probability measure on $\cA$ defined by
 $$E[y](a) = \#\{1\le i \le N:~ y(i) = a\}/N.$$
 By the law of large numbers we may choose $N$ large enough so that
 $$\pi^N(\{ y \in \cA^N:~ \| E[y] - \pi \|_\infty > \epsilon\}) < \epsilon^2/2.$$
 Therefore,
$$\mu\left(\left\{x \in X:~ \| E[\psi_{s,N}(x)] - \pi \|_\infty > \epsilon\right\}\right) = (\psi_{s,N})_*\mu(\{ y \in \cA^N:~ \| E[y] - \pi \|_\infty > \epsilon\})< \epsilon^2$$
for every $s\in S$. Let $Z=\left\{x \in X:~ \| E[\psi_{s,N}(x)] - \pi \|_\infty > \epsilon\right\}$. So 
$$\int \nu(Z)~d\omega_s(\nu) = \mu(Z)<\epsilon^2.$$
This implies $\omega_s(\{\nu:~ \nu(Z)>\epsilon\})<\epsilon$. 

Next we claim that if a probability measure $\nu$ satisfies $\nu(Z) \le \epsilon$ then  $\| \psi_*\nu - \pi \|_\infty \le 3\epsilon$. Indeed, 
$$\psi_*\nu = E_* (\psi_{s,N})_*\nu.$$
So if $\nu(Z)\le \epsilon$ then $\|\psi_* \nu - E_*(\psi_{s,N})_*(\nu\upharpoonright Z^c) \|_\infty \le \epsilon$ (where $Z^c = X \setminus Z$ is the complement of $Z$) and $\| \pi - E_*(\psi_{s,N})_*(\nu\upharpoonright Z^c) \|_\infty \le \frac{\epsilon}{1-\epsilon}$ by definition of $Z$. So $\| \psi_* \nu - \pi\|_\infty \le 3\epsilon$ (since we assume $\epsilon < 1/2$). 

Since $\omega_s(\{\nu:~ \nu(Z)>\epsilon\})<\epsilon$, we now have
$$\omega_s(\{\nu:~ \|\psi_*\nu - \pi\|_\infty > 3\epsilon\}) < \epsilon.$$


\end{proof}





\begin{proof}[Proof of Theorem \ref{thm:main}]
Let $G$ be a finitely generated free group with free generating set $S \subset G$. Let $a,b\in A(G,X,\mu)$ and assume $a$ is essentially free.  It suffices to show that $b\in \overline{[a]_{OE}}$. By Lemma \ref{lem:generator} it suffices to show that for every finite set $\cA$, measurable map $\phi:X \to \cA$ and $\epsilon>0$ there exists a measurable map $\psi:X \to \cA$ and $a'\in [a]_{OE}$ such that 
$$\| (\psi \vee \psi \circ a'_s)_*\mu - (\phi \vee \phi \circ b_s)_*\mu \|_\infty \le 10|\cA|\epsilon\quad \forall s\in S$$ 
where, for example, $\phi \vee \phi \circ b_s:X \to \cA \times \cA$ is the map
$$\phi\vee \phi \circ b_s(x) = (\phi(x), \phi(b_sx)).$$

After replacing $\cA$ with the essential range of $\phi$ is necessary, we may assume that $\phi_*\mu(c)>0$ for every $c\in \cA$. We claim that there exists a self-coupling $\cJ_s$ of $\phi_*\mu$ such that $\cJ_s(c,d) >0$ for all $c,d \in \cA$ and
$$\|(\phi \vee \phi \circ b_s)_*\mu  - \cJ_s\|_\infty < \epsilon.$$
Indeed, the self-coupling
$$\cJ_s = (1-\epsilon) (\phi \vee \phi \circ b_s)_*\mu + \epsilon (\phi \times \phi)_*(\mu \times \mu)$$
has this property. After choosing $\epsilon$ smaller if necessary we may assume that $\epsilon<1/6$ and 
$$\min_{s\in S} \min_{c,d \in \cA} \cJ_s(c,d) > 2|\cA|\epsilon.$$

Let $\mu=\int \nu~d\omega_s$ be the ergodic decomposition of $\mu$ with respect to $a_s$. By Lemma \ref{lem:good-partition} there exists a measurable map $\psi:X \to \cA$ such that
$$\omega_s(\{\nu:~\| \phi_*\mu - \psi_* \nu\|_\infty > \epsilon\})< \epsilon$$
for every $s\in S$. By Lemma \ref{lem:hard} for every $s\in S$ there exists $a'_s \in \Aut(X,\mu)$ such that $a'_s$ and $a_s$ have the same orbits and
$$\| (\psi \vee \psi \circ a'_s)_*\mu - \cJ_s \|_\infty \le 9|\cA|\epsilon.$$ 
Because $a'_s$ and $a_s$ have the same orbits for every $s\in S$ it follows that the homomorphism $a':G \to \Aut(X,\mu)$ defined by $\{a'_s\}_{s\in S}$ is orbit-equivalent to $a$. In other words, $a' \in [a]_{OE}$. Also
$$\| (\psi \vee \psi \circ a'_s)_*\mu - (\phi \vee \phi \circ b_s)_*\mu\|_\infty \le \|(\psi \vee \psi \circ a'_s)_*\mu - \cJ_s \|_\infty  + \|\cJ_s - (\phi \vee \phi \circ b_s)_*\mu\|_\infty \le 10|\cA|\epsilon.$$ 
This proves the special case in which $G$ is finitely generated. The general case follows from Corollary \ref{cor:fin-gen}.
\end{proof}

 \end{document}